\documentclass[12pt]{amsart}

\usepackage{amsmath, amsthm, amsfonts, amssymb}
\usepackage{mathrsfs}

\usepackage[margin=25mm]{geometry}

\newtheorem{proposition}{Proposition}
\newtheorem{theorem}{Theorem}
\newtheorem{cor}{Corollary}
\newtheorem{lemma}{Lemma}

\newtheorem{definition}{Definition}

\newcommand{\ze}{\mathbb{Z}}
\newcommand{\re}{\mathbb{R}}

\def\esp{\mathbb{E}}

\newcommand{\wh}{\mathscr{W}}

\newcommand{\dx}{\partial_x}
\newcommand{\ddx}{\partial^2_x}

\newcommand{\field}{\mathcal{X}}
\newcommand{\mart}{\mathcal{M}}
\newcommand{\spart}{\mathcal{S}}
\newcommand{\apart}{\mathcal{B}}
\newcommand{\mpart}{\widetilde{\mathcal{B}}}

\newcommand{\gradn}{\nabla^n}

\newcommand{\phin}{\varphi^n}
\newcommand{\gradphin}{\gradn \phin}

\newcommand{\cale}{\mathcal{E}}

\newcommand{\arru}{\overrightarrow{u}^l}

\def \simless {\mathbin{\lower 3pt\hbox{$\rlap{\raise 5pt
              \hbox{$\char'074$}}\mathchar"7218$}}}

\title[Scaling of the Sasamoto-Spohn model]{Scaling of the Sasamoto-Spohn model in equilibrium}

\author[Milton Jara and Gregorio R. Moreno Flores]{Milton Jara$^1$ and Gregorio R. Moreno Flores$^2$}

\thanks{ AMS 2000 {\it subject classifications}. Primary  60K35
 secondary 82B20, 60H15
 }


\thanks{$^1$ Instituto de Matem\'atica Pura e Aplicada. Partially supported by CNPq and FAPERJ}

\thanks{$^2$ Pontificia Universidad Cat\'olica de Chile. Partially supported by Fondecyt grant 1171257 and N\'ucleo Milenio `Modelos Estoc\'asticos de Sistemas Complejos y Desordenados'}

\thanks{This project has received funding from the European Research Council  
(ERC) under the European Union’s Horizon 2020 research and innovative  
programme (grant agreement No 715734)}

\address[Milton Jara]{Instituto de Matem\'atica Pura e Aplicada, Estrada Dona Castorina 110, 22460320 Rio de Janeiro, Brazil}

\email{mjara@impa.br}

\address[Gregorio R. Moreno Flores]{Facultad de Matem\'aticas\\
Pontificia Universidad Cat\'olica de Chile\\
Vicu\~na Mackenna 4860, Macul\\
Santiago, Chile}

\email{grmoreno@mat.uc.cl}

\email{}

\date{}

\begin{document}
	
\begin{abstract} 
	We prove the convergence of the Sasamoto-Spohn model in equilibrium to the energy solution of the stochastic Burgers equation on the whole line. The proof, which relies on the second order Boltzmann-Gibbs principle, follows the approach of \cite{GJS}  and does not use any spectral gap argument.	
\end{abstract}

\maketitle
%\tableofcontents

%%%%%%%%%%%%%%%%%%%%%%%%%%%%%%%%%%%%%%%%%%%%%%%%%%%%%%%%
%%%%%%%%%%%%%%%%%%%%%%%%%%%%%%%%%%%%%%%%%%%%%%%%%%%%%%%%
%%%%%%%%%%%%%%%%%%%%%%%%%%%%%%%%%%%%%%%%%%%%%%%%%%%%%%%%
%%%%%%%%%%%%%%%%%%%%%%%%%%%%%%%%%%%%%%%%%%%%%%%%%%%%%%%%
%%%%%%%%%%%%%%%%%%%%%%%%%%%%%%%%%%%%%%%%%%%%%%%%%%%%%%%%
%%%%%%%%%%%%%%%%%%%%%%%%%%%%%%%%%%%%%%%%%%%%%%%%%%%%%%%%
\section{Model and results}
The goal of this note is to show the convergence of a certain discretization of the stochastic Burgers equation:
\begin{eqnarray}\label{eq:burgers}
	\partial_t u = \tfrac12 \ddx u + \partial_x u^2+\partial_x \wh,
\end{eqnarray}
where $\wh$ is a space-time white noise.
This equation can be seen as the evolution of the slope of solutions to the KPZ equation \cite{KPZ} which is itself a model of an interface in a disordered environment. The KPZ/Burgers equation has been subject to an extensive body of work in the last years.  It appears as the scaling limit of a wide range of particle systems \cite{BG, GJ-arma}, directed polymer models \cite{AKQ, MQR} and interacting diffusions \cite{DGP}, and constitutes a central element in a vast family of models known as the KPZ universality class \cite{C-review, Q-review}. 

Due to the nonlinearity, a lot of care has to be taken to obtain a notion of solution for \eqref{eq:burgers}. There are today several alternatives, for instance, regularity structure \cite{H}, paracontrolled distributions \cite{GP-reloaded} and energy solutions  \cite{GJ-arma, GuJ, GP-uniqueness}, which is the approach we will follow. 

The discretization we consider corresponds to	
\begin{eqnarray}\label{eq:main-eq}
	d u_j &=& \tfrac12\Delta u_j + \gamma B_j(u) + d\xi_j-d\xi_{j-1},
\end{eqnarray}
where $(\xi_j)_j$ is an i.i.d. family of standard one-dimensional Brownian motions,
\begin{eqnarray*}
	\Delta u_j &=& u_{j+1}+u_{j-1}-2u_j,\\
	B_j(u)&=&w_j-w_{j-1} \quad \text{with} \quad w_j =\frac13( u^2_j+u_ju_{j+1}+u^2_{j+1} ).
\end{eqnarray*}
This model, introduced in \cite{KS} (see also \cite{LS}) and further studied in \cite{SS}, is nowadays often referred to as the Sasamoto-Spohn model.

While the discretization of the second derivative and noise are quite straightforward, there are a priori several ways to discretize the nonlinearity in Burgers equation. This particular choice is motivated by two reasons: first, it only involves nearest neighbor sites and, second, it yields the explicit invariant measure $\mu=\rho^{\otimes \ze}$, where
$d\rho(x)=\frac{1}{\sqrt{2\pi}}e^{-x^2/2}dx$ (see Section \ref{sec:generator}).

Our result states the convergence of the discrete equations \eqref{eq:main-eq} to Burgers equation in the sense of energy solutions (see Section \ref{sec:energy} for a precise definition).
\begin{theorem}
For each $n\geq 1$, let $u^n$ be the solution to the system \eqref{eq:main-eq} for $\gamma=n^{-1/4}$ and initial law $\mu$, and let
\begin{eqnarray*}
	\field^n_t(\varphi) = \frac{1}{n^{1/4}} \sum_j u^n_j(tn) \varphi(\tfrac{j}{\sqrt{n}}).
\end{eqnarray*}
The sequence of processes $(\field^n_{\cdot})_{n\geq 1}$ converges in distribution in $C([0,T],\mathcal{S}'(\re))$
to the unique energy solution of the Burgers equation.
\end{theorem}

A similar result was shown in \cite{GP-reloaded} for much more general initial conditions although restricted to the one-dimensional torus.

At the technical level, our approach relies on the techniques of \cite{GJS} and avoids the use of any spectral gap estimate. The core of the proof consists in deriving certain dynamical estimates among which the so-called second order Boltzmann-Gibbs principle plays a major role. A key ingredient is a certain integration-by-parts satisfied by the model. 

The paper is organized as follows: in Section \ref{sec:energy}, we recall the notion of energy solution from \cite{GJ-arma}. We show the invariance of the measure $\mu$ in Section \ref{sec:generator}. In Section \ref{sec:BG}, we prove the dynamical estimates. Finally, in Sections \ref{sec:tightness} and \ref{sec:convergence}, we show, respectively, tightness and convergence to the energy solution. The construction of the dynamics \eqref{eq:main-eq} is given in the appendix. 

%%%%%%%%%%%%%%%%%%%%%%%%%%%%%%%%%%%%%%%%%%%%%%%%%%%%%%%%%%
%%%%%%%%%%%%%%%%%%%%%%%%%%%%%%%%%%%%%%%%%%%%%%%%%%%%%%%%%%
%%%%%%%%%%%%%%%%%%%%%%%%%%%%%%%%%%%%%%%%%%%%%%%%%%%%%%%%%%
%%%%%%%%%%%%%%%%%%%%%%%%%%%%%%%%%%%%%%%%%%%%%%%%%%%%%%%%%%
%%%%%%%%%%%%%%%%%%%%%%%%%%%%%%%%%%%%%%%%%%%%%%%%%%%%%%%%%%
%%%%%%%%%%%%%%%%%%%%%%%%%%%%%%%%%%%%%%%%%%%%%%%%%%%%%%%%%%
\vspace{2ex}

\noindent \textsc{Notations:} We denote by $\mathcal{S}(\re)$ the space of Schwarz functions on $\re$. For $n\geq 1$ and a smooth function $\varphi$, we define $\phin_j = \varphi(\tfrac{j}{\sqrt{n}})$, $\gradphin_j=\sqrt{n}(\phin_{j+1}-\phin_j)$ and $\Delta^n \phin_j = n (\phin_{j+1}+\phin_{j-1}-2\phin)$. We also define
\begin{eqnarray*}
	\cale(\varphi) = \int \varphi^2(x)\, dx,\quad \cale_n(\psi) = \frac{1}{\sqrt{n}}\sum_{j\in\ze}\psi_j^2,
\end{eqnarray*}
respectively, for $\varphi\in L^2(\re)$ and $\psi \in l^2(\ze)$.

%%%%%%%%%%%%%%%%%%%%%%%%%%%%%%%%%%%%%%%%%%%%%%%%%%%%%%%%%%
%%%%%%%%%%%%%%%%%%%%%%%%%%%%%%%%%%%%%%%%%%%%%%%%%%%%%%%%%%
%%%%%%%%%%%%%%%%%%%%%%%%%%%%%%%%%%%%%%%%%%%%%%%%%%%%%%%%%%
%%%%%%%%%%%%%%%%%%%%%%%%%%%%%%%%%%%%%%%%%%%%%%%%%%%%%%%%%%
%%%%%%%%%%%%%%%%%%%%%%%%%%%%%%%%%%%%%%%%%%%%%%%%%%%%%%%%%%
%%%%%%%%%%%%%%%%%%%%%%%%%%%%%%%%%%%%%%%%%%%%%%%%%%%%%%%%%%

\section{Energy solutions of the Burgers equation}\label{sec:energy}
We will introduce the notion of an energy solution for Burgers equation \cite{GJ-arma}. We start with two definitions:

\begin{definition}
	We say that a process $\{u_t:\, t\in[0,T]\}$ satisfies condition (S) if, for all $t\in[0,T]$, the $\mathcal{S}'(\re)$-valued random variable $u_t$ is a white noise of variance $1$.
\end{definition}

For a stationary process $\{u_t:\, t\in[0,T]\}$, $0\leq s<t\leq T$, $\varphi\in\mathcal{S}(\re)$ and $\varepsilon>0$, we define
\begin{eqnarray*}
	\mathcal{A}^{\varepsilon}_{s,t}(\varphi) = \int^t_s \int_{\re} u_r(i_{\varepsilon}(x))^2 \dx\varphi(x)dxdr
\end{eqnarray*}
where $i_{\varepsilon}(x)=\varepsilon^{-1}{\bf 1}_{(x,x+\varepsilon]}$

\begin{definition}
	Let $\{u_t:\, t\in[0,T]\}$ be a process satisfying condition (S). We say that $\{u_t:\, t\in[0,T]\}$ satisfies the energy estimate if there exists a constant $\kappa>0$ such that:
	
	\vspace{1ex}
	
	\noindent (EC1) For any $\varphi\in\mathcal{S}(\re)$ and any $0\leq s<t\leq T$,
							\begin{eqnarray*}
								\esp\left[ \left| \int^t_s u_r(\ddx\varphi)\,dr\right|^2\right] \leq \kappa (t-s)\cale(\dx \varphi)
							\end{eqnarray*}

	\noindent (EC2) For any $\varphi\in\mathcal{S}(\re)$, any $0\leq s<t\leq T$ and any $0<\delta<\varepsilon <1$,
							\begin{eqnarray*}
								\esp\left[ \left| \mathcal{A}^{\varepsilon}_{s,t}(\varphi)-\mathcal{A}^{\delta}_{s,t}(\varphi) \right|^2\right] \leq \kappa (t-s)\varepsilon\cale(\dx \varphi)
							\end{eqnarray*}
\end{definition}
We state a theorem proved in \cite{GJ-arma}:
\begin{theorem}\label{thm:S-EC-imply}
	Assume $\{u_t:\, t\in[0,T]\}$ satisfies (S) and (EC2).
	There exists an $\mathcal{S}'(\re)$-valued stochastic process $\{\mathcal{A}_t:\, t\in[0,T]\}$ with continuous paths such that
	\begin{eqnarray*}
		\mathcal{A}_t(\varphi) = \lim_{\varepsilon\to0}\mathcal{A}^{\varepsilon}_{0,t}(\varphi),
	\end{eqnarray*}
	in $L^2$, for any $t\in[0,T]$ and $\varphi\in\mathcal{S}(\re)$.
\end{theorem}
We are now ready to formulate the definition of an energy solution:

\begin{definition}
	We say that $\{u_t:\, t\in[0,T]\}$ is a stationary energy solution of the Burgers equation if
	\begin{itemize}
		\item $\{u_t:\, t\in[0,T]\}$ satisfies (S), (EC1) and (EC2).
		
		\item For all $\varphi\in\mathcal{S}(\re)$, the process
				\begin{eqnarray*}
					u_t(\varphi)-u_0(\varphi)-\tfrac12 \int^t_0 u_s(\ddx \varphi)\,ds- \mathcal{A}_t(\varphi)
				\end{eqnarray*}
				is a martingale with quadratic variation $t\cale(\dx\varphi)$, where $\mathcal{A}$ is the process from Theorem \ref{thm:S-EC-imply}.
	\end{itemize}
\end{definition}
Existence of energy solutions was proved in \cite{GJ-arma}. Uniqueness was proved in \cite{GP-uniqueness}.

%%%%%%%%%%%%%%%%%%%%%%%%%%%%%%%%%%%%%%%%%%%%%%%%%%%%%%%%
%%%%%%%%%%%%%%%%%%%%%%%%%%%%%%%%%%%%%%%%%%%%%%%%%%%%%%%%
%%%%%%%%%%%%%%%%%%%%%%%%%%%%%%%%%%%%%%%%%%%%%%%%%%%%%%%%
%%%%%%%%%%%%%%%%%%%%%%%%%%%%%%%%%%%%%%%%%%%%%%%%%%%%%%%%
%%%%%%%%%%%%%%%%%%%%%%%%%%%%%%%%%%%%%%%%%%%%%%%%%%%%%%%%
%%%%%%%%%%%%%%%%%%%%%%%%%%%%%%%%%%%%%%%%%%%%%%%%%%%%%%%%

\section{Generator and invariant measure}\label{sec:generator}
The construction of the dynamics given by \eqref{eq:main-eq} is detailed in Appendix \ref{sec:construction}. We denote by $\mathscr{C}$ the set of cylindrical functions $F$ of the form $F(u) = f(u_{-n},\cdots, u_n)$,
for some $n\geq 0$, with $f\in C^2(\re^{2n+1})$ with polynomial growth of its partial derivatives up to order $2$.
The generator of the dynamics \eqref{eq:main-eq} acts on $\mathscr{C}$  as
\begin{eqnarray*}
	L = \sum_{j}\left\{
			\frac12 (\partial_{j+1}-\partial_j)^2-\frac12(u_{j+1}-u_j)(\partial_{j+1}-\partial_j)+ \gamma B_j(u)\partial_j
		\right\},
\end{eqnarray*}
where $\partial_j = \frac{\partial}{\partial u_j}$. Let us introduce the operators
\begin{eqnarray*}
	S &=& \sum_{j}\left\{
			\frac12 (\partial_{j+1}-\partial_j)^2-\frac12(u_{j+1}-u_j)(\partial_{j+1}-\partial_j)
		\right\}, \quad
	A =\sum_{j}\gamma B_j(u)\partial_j,
\end{eqnarray*}
which formally correspond to the symmetric and anti-symmetric parts of $L$ with respect to $\mu=\rho^{\otimes \ze}$, where
$d\rho(x)=\frac{1}{\sqrt{2\pi}}e^{-x^2/2}dx$. We note that our model satisfies the Gaussian integration-by-parts formula:
%\begin{eqnarray}
%	\int (u_{j+1}-u_j) f d\mu = \int(\partial_{j+1}-\partial_j) f d\mu, 
%\end{eqnarray}
\begin{eqnarray*}
	\int u_j f d\mu = \int\partial_j f d\mu, 
\end{eqnarray*}
which will be heavily used in the sequel.

We will also consider the periodic model $u^M$ on $\ze_M:=\ze \slash M\ze$ and denote by $L_M,\, S_M$ and $A_M$ the corresponding generator and its symmetric and anti-symmetric parts respectively. Finally, denote $\mu_M = \rho^{\otimes \ze_M}$ and let $\rho_M$ be its density.

\begin{lemma}
	The measure $\mu_M$ is invariant for the periodic dynamics $u^M$.
\end{lemma}
\begin{proof}
	The lemma follows from Echeverr\'ia's criterion (\cite{EK}, Thm 4.9.17) once we show
	\begin{eqnarray*}
		\int L_Mf\, d\mu_M = 0,
	\end{eqnarray*}
	for all $f\in C^2(\re^{\ze_M})$ with polynomial growth of its derivatives up to order $2$.
	By standard integration-by-parts,
	\begin{eqnarray*}
		\int S_Mf\, d\mu_M %&=& \int S_Mf(u)\rho_M(u) \,du_{-M}\cdots du_M\\
				&=& \int f(u) S_M^{\dagger} \rho_M(u)\,du_{-M}\cdots du_M,
	\end{eqnarray*}
	where
	\begin{eqnarray*}
		S^{\dagger}_M = \frac12 \sum_{j\in\ze_M} \left\{ (\partial_{j+1}-\partial_j)^2 + (u_j-u_{j+1})(\partial_j-\partial_{j+1})+2\right\}.
	\end{eqnarray*}
	It is a simple computation to show that $S^{\dagger}_M\rho_M \equiv 0$. It then remains to verify that 
	\begin{eqnarray*}
		\int A_Mf\, d\mu_M = \int \sum_{j\in \ze_M}(w_j-w_{j-1})\partial_j f(u)\rho_M(u)\, du_{-M}\cdots du_M= 0.
	\end{eqnarray*}
	But, using standard integration-by-parts once again, we can verify that there exists a degree three polynomial in two variables $p(\cdot,\cdot)$ such that
	\begin{eqnarray*}
		\int A_Mf\, d\mu_m = \int \sum_{j\in\ze_M} f(u) \left\{ p(u_j,u_{j+1})-p(u_{j-1},u_{j})\right\}d\mu_M.
	\end{eqnarray*}
	Finally, Gaussian integration-by-parts yields a degree two polynomial in two variables $\tilde p(\cdot,\cdot)$ such that
	\begin{eqnarray*}
		\int A_Mf\, d\mu_M = \int \sum_{j\in\ze_M}
		\left\{ \tilde p(\partial_j,\partial_{j+1})-\tilde p(\partial_{j-1},\partial_{j})\right\}f(u)\, d\mu,
	\end{eqnarray*}
	which is telescopic. This ends the proof.
\end{proof}
By construction of the infinite volume dynamics and taking the limit $M\to\infty$, we obtain
\begin{cor}
	The measure $\mu$ is invariant for the dynamics \eqref{eq:main-eq}.
\end{cor}

%%%%%%%%%%%%%%%%%%%%%%%%%%%%%%%%%%%%%%%%%%%%%%%%%%%%%%%%%%
%%%%%%%%%%%%%%%%%%%%%%%%%%%%%%%%%%%%%%%%%%%%%%%%%%%%%%%%%%
%%%%%%%%%%%%%%%%%%%%%%%%%%%%%%%%%%%%%%%%%%%%%%%%%%%%%%%%%%
%%%%%%%%%%%%%%%%%%%%%%%%%%%%%%%%%%%%%%%%%%%%%%%%%%%%%%%%%%
%%%%%%%%%%%%%%%%%%%%%%%%%%%%%%%%%%%%%%%%%%%%%%%%%%%%%%%%%%
%%%%%%%%%%%%%%%%%%%%%%%%%%%%%%%%%%%%%%%%%%%%%%%%%%%%%%%%%%

\section{The second-order Boltzmann-Gibbs principle}\label{sec:BG}
We recall the Kipnis-Varadhan inequality: there exists $C>0$ such that
% (see, for instance, \cite{KLO}, Lemma 2.4):
\begin{eqnarray}\label{eq:KV}
	\esp\left[\sup_{0\leq t \leq T} \left| \int^t_0 F(u(sn))\, ds\right|^2\right]
	\leq C T ||F(\cdot)||_{-1,n}^2ds,
\end{eqnarray}
where the $||\cdot||_{-1,n}$-norm is defined through the variational formula
\begin{eqnarray*}
	||F||_{-1,n}^2 = \sup_{f\in \mathscr{C}}\left\{ 2 \int F(u)fd\mu+n \int f Lf d\mu\right\}
\end{eqnarray*}
The proof of this inequality in our context follows from a straightforward modification of the arguments of \cite{GP-uniqueness}, Corollary 3.5.
In our particular model, we have
\begin{eqnarray*}
	-\int f Lf d\mu = \frac{1}{2}\sum_j \int \left( (\partial_{j+1}-\partial_j)f\right)^2d\mu
\end{eqnarray*}
so that the variational formula becomes
\begin{eqnarray*}
	||F||_{-1,n}^2 = \sup_{f\in \mathscr{C}}\left\{ 2 \int F(u)fd\mu-\frac{n}{2}\sum_j \int \left( (\partial_{j+1}-\partial_j)f\right)^2d\mu\right\}.
\end{eqnarray*}
Denote by $\tau_j$ the canonical shift $\tau_j u_{i}=u_{j+i}$ and let $\arru_j = \frac{1}{l}\sum^l_{k=1}u_{j+k}$.
%\begin{eqnarray}
%	\arru_j = \frac{1}{l}\sum^l_{k=1}u_{j+k}.
%\end{eqnarray}

\begin{lemma}\label{thm:one-block} Let $l\geq 1$ and let $g$ be a function with zero mean with respect to $\mu$ which support does not intersect $\{1,\cdots,l\}$. Let $g_j(s) = g(\tau_j u(s))$. There exists a constant $C>0$ such that
	\begin{eqnarray}\label{eq:one-block}
		\esp\left[ \left| \int^t_0ds\sum_j g_j(sn)[u_{j+1}(sn)-\arru_j(sn)]\varphi_j\right|^2\right]
		\leq
		C \frac{tl}{\sqrt{n}}||g||^2_{L^2(\mu)}\cale_n(\varphi)
	\end{eqnarray}
\end{lemma}
\begin{proof}
	Let $\psi_i = \frac{l-i}{l}$, $i=0,\cdots, l-1$. Then,
	\begin{eqnarray*}
		u_{j+1}-\arru_j = \sum^{l-1}_{i=1}(u_{j+i}-u_{j+i+1})\psi_i.
	\end{eqnarray*}
	Hence,
	\begin{eqnarray*}
		\sum_j \varphi_j g_j (u_{j+1}-\arru_j)
		&=&
		\sum_j \varphi_j g_j\sum^{l-1}_{i=0}(u_{j+i}-u_{j+i+1})\psi_i\\
		&=&
		\sum_k \left( \sum^{l-1}_{i=1}\varphi_{k-i}g_{k-i}\psi_i\right)(u_k-u_{k+1})\\
		&=:&
		\sum_k F_k(u_k-u_{k+1})
	\end{eqnarray*}
	Now, for $f\in\mathscr{C}$, using integration-by-parts,
	\begin{eqnarray*}
		2 \int \sum_j \varphi_j g_j (u_{j+1}-\arru_j) f d\mu 
		&=&
		2\int \sum_k F_k(u_k-u_{k+1}) f d\mu\\
		&=&
		2\int \sum_k F_k (\partial_k-\partial_{k+1}) f d\mu\\
		&\leq&
		\int \sum_k\left\{ \alpha F_k^2 + \frac{1}{\alpha}((\partial_k-\partial_{k+1}) f)^2\right\}d\mu,
	\end{eqnarray*}
	by Young's inequality. Taking $\alpha=2/n$, we find that the above is bounded by
	\begin{eqnarray*}
		\frac{2}{n} \sum_k \int \sum_k F^2_k d\mu + \frac{n}{2} \sum_k \int((\partial_k-\partial_{k+1}) f)^2d\mu,
	\end{eqnarray*}
	which, thanks to the Kipnis-Varadhan inequality, shows that the left-hand-side of \eqref{eq:one-block} is bounded by
	\begin{eqnarray*}
%		|| \sum_j \varphi_j g_{j}(u_{j+1}-\arru_j) ||_{-1} 
%		\leq
		C\frac{t}{n} \sum_k \int F^2_k d\mu.
	\end{eqnarray*}
	Finally, as $g$ is centered,
	\begin{eqnarray*}
		\sum_k \int F^2_k d\mu \leq \sum_k \sum^{l-1}_{i=1} \varphi^2_{k-i} \int g^2d\mu \leq l \sqrt{n} \int g^2d\mu \cale_n(\varphi).
	\end{eqnarray*}
\end{proof}

%Let $Q^l(t,j)= (\arru_j(t))^2-\frac{1}{l}$.
We now state the second-order Boltzmann-Gibbs principle: let $Q(l,u)=(\arru_0)^2-\frac{1}{l}$,
\begin{proposition}\label{thm:2nd-BG} Let $l\geq 1$. There exists a constant $C>0$ such that
	\begin{eqnarray*}
		\esp\left[ \left|\int^t_0ds\sum_j \left\{ u_j(sn)u_{j+1}(sn)-\tau_jQ(l,u(sn))\right\} \varphi_j\right|^2\right]
		\leq
		C \frac{tl}{\sqrt{n}}\cale_n(\varphi)
	\end{eqnarray*}
\end{proposition}
\begin{proof}
	We use the factorization
	\begin{eqnarray*}
		u_j u_{j+1}-\tau_j Q(l,u)
		&=&
		u_j (u_{j+1}-\arru_j) + \arru_j (u_j-\arru_j)+\frac{1}{l}.
	\end{eqnarray*}
	We handle the first term with Lemma \ref{thm:one-block}. The second term is treated in the following lemma.
\end{proof}

\begin{lemma} Let $l\geq 1$. There exists a constant $C>0$ such that
	\begin{eqnarray*}
		\esp\left[ \left| \int^t_0ds\sum_j \left\{ \arru_j(sn)[u_j(sn)-\arru_j(sn)]+\frac{1}{l}\right\}\varphi_j\right|^2\right]
		\leq
		C \frac{tl}{\sqrt{n}}\cale_n(\varphi)
	\end{eqnarray*}
\end{lemma}
\begin{proof}
	Let $\psi_i=\frac{l-i}{l}$. Then,
	\begin{eqnarray*}
		\arru_j [u_j-\arru_j] = \sum^{l-1}_{i=0} \psi_i(u_{j+i}-u_{j+i+1})\arru_j.
	\end{eqnarray*}
	For $f\in\mathscr{C}$, using integration-by-parts,
	\begin{eqnarray*}
	\int \arru_j[u_j-\arru_j]fd\mu 
	&=&
	\int \sum^{l-1}_{i=0} \psi_i(u_{j+i}-u_{j+i+1})\arru_jfd\mu \\
	&=&
	\int \left\{\sum^{l-1}_{i=0} \psi_i \arru_j(\partial_{j+i}-\partial_{j+i+1})f-\frac{1}{l}f \right\}d\mu
	\end{eqnarray*}
	The second summand comes from the term $i=0$. Hence,
	\begin{eqnarray*}
		2\int \sum_j \varphi_j \left\{ \arru_j[u_j-\arru_j] + \frac{1}{l}\right\}f d\mu
		=
		2\int \sum_j \varphi_j\sum^{l-1}_{i=0} \psi_i \arru_j(\partial_{j+i}-\partial_{j+i+1})fd\mu
	\end{eqnarray*}
	By Young's inequality, this last expression is bounded by
	\begin{eqnarray*}
		&&\int \sum_j \sum^{l-1}_{i=0} \left\{
			\alpha \varphi_j^2(\arru_j)^2+\frac{1}{\alpha} \psi_i^2((\partial_{j+i}-\partial_{j+i+1})f)^2
				\right\}d\mu\\
		&\leq&
		\alpha l \int \sum_j  \varphi_j^2(\arru_j)^2 d\mu+\frac{l}{\alpha}\int \sum_j ((\partial_{j}-\partial_{j+1})f)^2d\mu 
	\end{eqnarray*}
	Taking $\alpha=2l/n$, this is further bounded by
	\begin{eqnarray*}
		&&\frac{2l^2}{n} \int (\arru_j)^2 d\mu \sum_j \varphi_j^2+\frac{n}{2}\int \sum_j ((\partial_{j}-\partial_{j+1})f)^2d\mu\\
		&\leq&
		\frac{l}{\sqrt{n}} \cale_n(\varphi)+\frac{n}{2} \int \sum_j ((\partial_{j}-\partial_{j+1})f)^2d\mu.
	\end{eqnarray*}
	The result then follows from the Kipnis-Varadhan inequality.
\end{proof}

%%%%%%%%%%%%%%%%%%%%%%%%%%%%%%%%%%%%%%%%%%%%%%%%%%%%%%%%
%%%%%%%%%%%%%%%%%%%%%%%%%%%%%%%%%%%%%%%%%%%%%%%%%%%%%%%%
%%%%%%%%%%%%%%%%%%%%%%%%%%%%%%%%%%%%%%%%%%%%%%%%%%%%%%%%
%%%%%%%%%%%%%%%%%%%%%%%%%%%%%%%%%%%%%%%%%%%%%%%%%%%%%%%%
%%%%%%%%%%%%%%%%%%%%%%%%%%%%%%%%%%%%%%%%%%%%%%%%%%%%%%%%
%%%%%%%%%%%%%%%%%%%%%%%%%%%%%%%%%%%%%%%%%%%%%%%%%%%%%%%%
\section{Tightness}\label{sec:tightness}
In the sequel, we let $\varphi\in \mathcal{S}$ be a test function. Remember the fluctuation field is given by
\begin{eqnarray*}
	\field^n_t(\varphi) = \frac{1}{n^{1/4	}}\sum_j u_j(nt)\phin_j.
\end{eqnarray*}
Recalling the definition of the operators $S$ and $A$ from Section \ref{sec:generator}, the symmetric and anti-symmetric parts of the dynamics are given by
\begin{eqnarray*}
	d\spart^n_t(\varphi) &= n S \field^n_t(\varphi) dt
								  &= \frac{1}{n^{1/4	}} n \sum_j u_j(tn) \Delta \phin_j dt= \frac{1}{n^{1/4	}} \sum_j u_j(tn) \Delta^n \phin_j dt\\
	d\apart^n_t(\varphi) &= n A \field^n_t(\varphi) dt
								   &= - \frac{1}{n^{1/2	}}n \sum_j w_j(tn) (\phin_{j+1}-\phin_j)dt =  \sum_j w_j (tn)\nabla^n \phin_j dt
\end{eqnarray*}
where we used $\gamma=n^{-1/4}$. Then, the martingale part of the dynamics corresponds to
\begin{eqnarray*}
	\mart^n_t(\varphi) &=& \field^n_t(\varphi) -\field^n_0(\varphi)-\spart^n_t(\varphi)-\apart^n_t(\varphi)
		= n^{1/4} \int^t_0 \sum_j (\varphi_{j}-\varphi_{j+1})d\xi_j(s)
\end{eqnarray*}
and has quadratic variation
\begin{eqnarray*}
	\langle \mart^n(\varphi) \rangle_t = n^{1/2} t\sum_j (\phin_{j}-\phin_{j+1})^2 = t \cale_n(\nabla^n \phin)
\end{eqnarray*}
We will use Mitoma's criterion \cite{Mitoma}: a sequence $\mathcal{Y}^n$ is tight in $C([0,T],\mathcal{S}'(\re))$ if and only if $\mathcal{Y}^n(\varphi)$ is tight in $C([0,T],\re)$ for all $\varphi\in\mathcal{S}(\re)$.
%%%%%%%%%%%%%%%%%%%%%%%%%%%%%%%%%%%%%%%%%%%%%%%%%%%%%%%%%%
%%%%%%%%%%%%%%%%%%%%%%%%%%%%%%%%%%%%%%%%%%%%%%%%%%%%%%%%%%
%%%%%%%%%%%%%%%%%%%%%%%%%%%%%%%%%%%%%%%%%%%%%%%%%%%%%%%%%%
%%%%%%%%%%%%%%%%%%%%%%%%%%%%%%%%%%%%%%%%%%%%%%%%%%%%%%%%%%
\subsection{Martingale term}
We recall that $\langle\mart^n(\varphi) \rangle = t \cale_n(\gradphin)$. From the Burkholder-Davis-Gundy inequality, it follows that
	\begin{eqnarray*}
		\esp\left[\left| \mart^n_t(\varphi)-\mart^n_s(\varphi)\right|^p\right]
		\leq C
		|t-s|^{p/2}\cale_n(\gradphin)^{p/2},
	\end{eqnarray*}
for all $p\geq 1$. Tightness then follows from Kolmogorov criterion by taking $p$ large enough.

%%%%%%%%%%%%%%%%%%%%%%%%%%%%%%%%%%%%%%%%%%%%%%%%%%%%%%%%%%
%%%%%%%%%%%%%%%%%%%%%%%%%%%%%%%%%%%%%%%%%%%%%%%%%%%%%%%%%%
%%%%%%%%%%%%%%%%%%%%%%%%%%%%%%%%%%%%%%%%%%%%%%%%%%%%%%%%%%
%%%%%%%%%%%%%%%%%%%%%%%%%%%%%%%%%%%%%%%%%%%%%%%%%%%%%%%%%%

\subsection{Symmetric term}
Tightness is obtained via a second moment computation and Kolmogorov criterion:
\begin{eqnarray*}
	\esp\left[\left| \spart^n_t(\varphi)-\spart^n_s(\varphi)\right|^2\right]
	&\leq&
	|t-s|^2 \frac{1}{\sqrt{n}} \sum_j \esp[u_j^2] (\Delta^n \phin_j)^2
	=|t-s|^2\cale_n(\Delta^n\phin).
\end{eqnarray*}
%%%%%%%%%%%%%%%%%%%%%%%%%%%%%%%%%%%%%%%%%%%%%%%%%%%%%%%%%%
%%%%%%%%%%%%%%%%%%%%%%%%%%%%%%%%%%%%%%%%%%%%%%%%%%%%%%%%%%
%%%%%%%%%%%%%%%%%%%%%%%%%%%%%%%%%%%%%%%%%%%%%%%%%%%%%%%%%%
%%%%%%%%%%%%%%%%%%%%%%%%%%%%%%%%%%%%%%%%%%%%%%%%%%%%%%%%%%

\subsection{Anti-symmetric term}
We study the tightness of the term
\begin{eqnarray*}
	\apart^n_t(\varphi) &=& \int^t_0 \sum_j w_j(sn) \gradphin_j ds\\
			&=& \int^t_0 \sum_j \frac13[u_{j+1}^2(sn)+u_j(sn)u_{j+1}(sn)+u_j^2(sn)]\gradphin_j ds.
\end{eqnarray*}

We begin with a lemma:
%%%%%%%%%%%%%%%%%%%%%%%%%%%%%%%%%%%%%%%%%%%%%%%%%%%%%%%%%%
\begin{lemma}
	The process
	\begin{eqnarray*}
		Y^n_t(\varphi)=\int^t_0ds \sum_j\varphi_j \left\{(u_j(sn)u_{j+1}(sn)-u_j^2(sn))+1\right\}
	\end{eqnarray*}
	goes to zero in the ucp topology.
\end{lemma}
%%%%%%%%%%%%%%%%%%%%%%%%%%%%%%%%%%%%%%%%%%%%%%%%%%%%%%%%%%
\begin{proof}
Using integration by parts,
\begin{eqnarray*}
	\int \sum_j\varphi_j (u_ju_{j+1}-u_j^2) fd\mu
	&=&
	\int \sum_j\varphi_j (u_{j+1}-u_j) u_j fd\mu\\
	&=&
	\int \sum_j\varphi_j (\partial_{j+1}-\partial_j)(u_j f)d\mu\\
	&=&
	\int \sum_j\varphi_j \left\{u_j (\partial_{j+1}-\partial_j) f-f\right\}
\end{eqnarray*}
Hence,
\begin{eqnarray*}
	\int \sum_j\varphi_j \left\{(u_ju_{j+1}-u_j^2)+1\right\} fd\mu
	&=&
	\int \sum_j\varphi_j u_j (\partial_{j+1}-\partial_j) fd\mu
\end{eqnarray*}
%The constant can actually be ignored as we test against gradients. 
Using Young's inequality,
\begin{eqnarray*}
	2\int \sum_j\varphi_j \left\{(u_ju_{j+1}-u_j^2)+1\right\} fd\mu
	&\leq&
	\int \sum_j\left\{ \alpha \varphi_j^2 u_j^2+ \frac{1}{\alpha} ((\partial_{j+1}-\partial_j) f)^2\right\}d\mu\\
	&\leq&
	\frac{2}{\sqrt{n}} \mathcal{E}_n(\varphi) + \frac{n}{2}\sum_j \int ((\partial_{j+1}-\partial_j) f)^2d\mu,
\end{eqnarray*}
by taking $\alpha=2/n$. Into the Kipnis-Varadhan inequality, this yields
\begin{eqnarray*}
	\esp\left[\sup_{0\leq t \leq T}\left| \int^t_0ds \sum_j\varphi_j \left\{(u_j(sn)u_{j+1}(sn)-u_j^2(sn))+1\right\}\right|^2\right]
	&\leq&
	\frac{C T}{\sqrt{n}} \mathcal{E}_n(\varphi) 
\end{eqnarray*}
which shows that this process goes to zero in the ucp topology.
\end{proof}
%%%%%%%%%%%%%%%%%%%%%%%%%%%%%%%%%%%%%%%%%%%%%%%%%%%%%%%%%%
This means we can switch the term $w_j$ in the anti-symmetric part of the dynamics by $u_j u_{j+1}$ modulo a vanishing term. Note that, as we apply the previous lemma to a gradient, the constant term $1$ will disappear. We are then left to prove the tightness of
\begin{eqnarray*}
	\mpart^n_t(\varphi) &=& \int^t_0 \sum_j u_j(sn)u_{j+1}(sn)\gradphin_jds.
\end{eqnarray*}
From Proposition \ref{thm:2nd-BG}, we have
\begin{eqnarray*}
	\esp\left[ \left| \mpart^n_t(\varphi)-\int^t_0 \sum_j \tau_j Q(l,u(sn))\gradphin_j ds \right|^2\right]
	\leq C \frac{tl}{\sqrt{n}} \cale_n(\gradphin)
\end{eqnarray*}
where, here and below, $C$ denotes a constant which value can change from line to line.
On the other hand, a careful $L^2$ computation, taking dependencies into account, shows that
\begin{eqnarray*}
\esp\left[ \left| \int^t_0 \sum_j \tau_j Q(l,u(sn))\gradphin_j ds\right|^2\right]
%	&\leq&
%	t^2 \esp\left[ \left| \sum_j u_j\arru_j\gradphin_j \right|^2\right]\\
%	&=&
%	t^2 \esp\left[ \left| \sum^l_{k=0}\sum_j u_{k+(l+1)j}\arru_{k+(l+1)j}\gradphin_{k+(l+1)j} \right|^2\right]\\
%	&\simless&
%	t^2  \sum^l_{k=0} \esp\left[ \left|\sum_j u_{k+(l+1)j}\arru_{k+(l+1)j}\gradphin_{k+(l+1)j} \right|^2\right]\\
%	&=&
%	t^2 \sum^l_{k=0} \sum_j (\gradphin_{k+(l+1)j})^2 \esp\left[ u_j^2 (\arru_j)^2\right]\\
%	&=&
%	\frac{t^2}{l}\sum^l_{k=0} \sum_j (\gradphin_{k+(l+1)j})^2\\
	&\leq&
	C
	\frac{t^2\sqrt{n}}{l}\cale_n(\gradphin).
\end{eqnarray*}
Observe that $\lim_{n\to \infty}\cale_n(\gradphin)=\int \partial_x\varphi(x)^2\,dx<\infty$.
Summarizing,
\begin{eqnarray*}
	\esp\left[ \left|\mpart^n_t(\varphi) \right|^2\right] \leq C \left\{ \frac{tl}{\sqrt{n}}+\frac{t^2\sqrt{n}}{l}\right\}.
\end{eqnarray*}
For $t\geq 1/n$, we take $l\sim \sqrt{tn}$ and get
\begin{eqnarray*}
	\esp\left[ \left|\mpart^n_t(\varphi) \right|^2\right] \leq C t^{3/2}.
\end{eqnarray*}
For $t\leq 1/n$, a crude $L^2$ bound gives
\begin{eqnarray*}
	\esp\left[ \left|\mpart^n_t(\varphi) \right|^2\right] \leq C t^2 \sqrt{n}\leq C t^{3/2}.
\end{eqnarray*}
This gives tightness.
%%%%%%%%%%%%%%%%%%%%%%%%%%%%%%%%%%%%%%%%%%%%%%%%%%%%%%%%%%
%%%%%%%%%%%%%%%%%%%%%%%%%%%%%%%%%%%%%%%%%%%%%%%%%%%%%%%%%%
%%%%%%%%%%%%%%%%%%%%%%%%%%%%%%%%%%%%%%%%%%%%%%%%%%%%%%%%%%
%%%%%%%%%%%%%%%%%%%%%%%%%%%%%%%%%%%%%%%%%%%%%%%%%%%%%%%%%%
%%%%%%%%%%%%%%%%%%%%%%%%%%%%%%%%%%%%%%%%%%%%%%%%%%%%%%%%%%
%%%%%%%%%%%%%%%%%%%%%%%%%%%%%%%%%%%%%%%%%%%%%%%%%%%%%%%%%%

\section{Convergence}\label{sec:convergence}
From the previous section, we get processes $\field$, $\spart$, $\apart$ and $\mart$ such that
%\begin{eqnarray}
%	\lim_{n\to\infty}\field^n &=& \field,\\
%	\lim_{n\to\infty} \spart^n &=& \spart,\\
%	\lim_{n\to\infty} \apart^n &=& \apart,\\
%	\lim_{n\to\infty} \mart^n &=& \mart.
%\end{eqnarray}
\begin{eqnarray*}
	\lim_{n\to\infty}\field^n &=& \field, \phantom{blal}\lim_{n\to\infty} \spart^n = \spart,\\
	\lim_{n\to\infty} \apart^n &=& \apart, \phantom{bla} \lim_{n\to\infty} \mart^n = \mart,
\end{eqnarray*}
along a subsequence that we still denote by $n$. We will now identify these limiting processes.

%%%%%%%%%%%%%%%%%%%%%%%%%%%%%%%%%%%%%%%%%%%%%%%%%%%%%%%%%%
%%%%%%%%%%%%%%%%%%%%%%%%%%%%%%%%%%%%%%%%%%%%%%%%%%%%%%%%%%
%%%%%%%%%%%%%%%%%%%%%%%%%%%%%%%%%%%%%%%%%%%%%%%%%%%%%%%%%%
%%%%%%%%%%%%%%%%%%%%%%%%%%%%%%%%%%%%%%%%%%%%%%%%%%%%%%%%%%

\subsection{Convergence at fixed times}
A straightforward  adaptation of the arguments in \cite{DGP}, Section 4.1.1, shows that $\field^n_t$ converges to a white noise for each fixed time $t\in[0,T]$. This in turns proves that the limit satisfies property (S).

%%%%%%%%%%%%%%%%%%%%%%%%%%%%%%%%%%%%%%%%%%%%%%%%%%%%%%%%%%
%%%%%%%%%%%%%%%%%%%%%%%%%%%%%%%%%%%%%%%%%%%%%%%%%%%%%%%%%%
%%%%%%%%%%%%%%%%%%%%%%%%%%%%%%%%%%%%%%%%%%%%%%%%%%%%%%%%%%
%%%%%%%%%%%%%%%%%%%%%%%%%%%%%%%%%%%%%%%%%%%%%%%%%%%%%%%%%%

\subsection{Martingale term}
	The quadratic variation of the martingale part satisfies
	\begin{eqnarray*}
		\lim_{n\to\infty} \langle\mart^n(\varphi)\rangle_t = t || \partial_x \varphi ||^2_{L^2}.
	\end{eqnarray*}
	By a criterion of Aldous \cite{A}, this implies convergence to the white noise.
%%%%%%%%%%%%%%%%%%%%%%%%%%%%%%%%%%%%%%%%%%%%%%%%%%%%%%%%%%
%%%%%%%%%%%%%%%%%%%%%%%%%%%%%%%%%%%%%%%%%%%%%%%%%%%%%%%%%%
%%%%%%%%%%%%%%%%%%%%%%%%%%%%%%%%%%%%%%%%%%%%%%%%%%%%%%%%%%
%%%%%%%%%%%%%%%%%%%%%%%%%%%%%%%%%%%%%%%%%%%%%%%%%%%%%%%%%%

\subsection{Symmetric term}
A second moment bound shows that
\begin{eqnarray*}
	\esp\left[\left| \spart^n_t(\varphi)-\int^t_0\field^n_s(\partial_x^2\varphi)\,ds\right|^2\right]\leq C\frac{t^2}{n},
\end{eqnarray*}
which shows that
\begin{eqnarray*}
	\spart(\varphi) = \lim_{n\to\infty}\spart^n(\varphi) =\int^{\cdot}_0 \field_s(\partial^2_x \varphi)\,ds.
\end{eqnarray*}

%%%%%%%%%%%%%%%%%%%%%%%%%%%%%%%%%%%%%%%%%%%%%%%%%%%%%%%%%%
%%%%%%%%%%%%%%%%%%%%%%%%%%%%%%%%%%%%%%%%%%%%%%%%%%%%%%%%%%
%%%%%%%%%%%%%%%%%%%%%%%%%%%%%%%%%%%%%%%%%%%%%%%%%%%%%%%%%%
%%%%%%%%%%%%%%%%%%%%%%%%%%%%%%%%%%%%%%%%%%%%%%%%%%%%%%%%%%

\subsection{Anti-symmetric term}

We just have to identify the limit of the process $\mpart^n(\varphi)$.
%\begin{eqnarray}
%	\mpart^n_t(\varphi) &=& \int^t_0 \sum_j u_j(sn)u_{j+1}(sn)\gradphin_jds
%\end{eqnarray}
Remembering the definition of the field $\field^n$, we observe that
\begin{eqnarray*}
	\sqrt{n}Q(\varepsilon\sqrt{n},u(nt)) = \field^n_t(i_{\varepsilon}(0))^2-\frac{1}{\varepsilon},
\end{eqnarray*}
from where we get the convergences
\begin{eqnarray*}
	\lim_{n\to\infty}\sqrt{n}Q(\varepsilon\sqrt{n},u(nt)) = \field_t(i_{\varepsilon}(0))^2-\frac{1}{\varepsilon}
\end{eqnarray*}
and
\begin{eqnarray*}
	\mathcal{A}^{\varepsilon}_{s,t}(\varphi) := \lim_{n\to\infty}\int^t_s \sum_j \tau_jQ(\varepsilon\sqrt{n},u(rn))\gradphin_j dr.
\end{eqnarray*}
The second limit follows by a suitable approximation of $i_{\varepsilon}(x)$ by $\mathcal{S}(\re)$ functions (see \cite{GJ-arma}, Section 5.3 for details). Now, by the second-order Boltzmann-Gibbs principle and stationarity,
\begin{eqnarray*}
	\esp\left[ \left| \mpart^n_t(\varphi)-\mpart^n_s(\varphi)-\int^t_s \sum_j \tau_jQ(l,u(rn))\gradphin_j dr\right|^2\right]
	\leq C
	\frac{(t-s)l}{\sqrt{n}}.
\end{eqnarray*}
Taking $l\sim\varepsilon\sqrt{n}$ and the limit as $n\to\infty$ along the subsequence,
\begin{eqnarray}\label{eq:BA}
	\esp\left[ \left| \apart_t(\varphi)-\apart_s(\varphi)-\mathcal{A}^{\varepsilon}_{s,t}(\varphi)\right|^2\right]
	\leq C
	(t-s)\varepsilon.
\end{eqnarray}
The energy estimate (EC2) then follows by the triangle inequality. Theorem \ref{thm:S-EC-imply} yields the existence of the process
\begin{eqnarray*}
	\mathcal{A}_t(\varphi) = \lim_{\varepsilon\to0}\mathcal{A}_{0,t}(\varphi).
\end{eqnarray*}
Furthermore, from \eqref{eq:BA}, we deduce that $\apart=\mathcal{A}$.

It remains to check (EC1). It is enough to check that
\begin{eqnarray*}
	\esp\left[ \left| \int^t_0 \field^n_s(\partial^2_x\varphi)\right|^2\right]\leq \kappa t.
\end{eqnarray*}
Using the smoothness of $\varphi$ and a summation by parts, it is further enough to verify that
\begin{eqnarray}\label{eq:ec2-1}
	\esp\left[ \left| \int^t_0 n^{1/4}\sum_j [u_{j+1}(sn)-u_j(sn)] \gradphin_j\right|^2\right]\leq \kappa t.
\end{eqnarray}
For that purpose, we will use Kipnis-Varadhan inequality one last time: let $f\in\mathscr{C}$,
\begin{eqnarray*}
	2\int n^{1/4} \sum_j (u_{j+1}-u_j)\gradphin_j f d\mu  
	&=&
	2\int n^{1/4} \sum_j \gradphin_j(\partial_{j+1}-\partial_j) f d\mu\\
	&\leq&
	\sum_j \left\{ \alpha\sqrt{n}(\gradphin_j)^2+\frac{1}{\alpha}\int ((\partial_{j+1}-\partial_j) f)^2d\mu \right\}\\
	&\leq&
	2 \cale_n(\gradphin) + \frac{n}{2}\sum_j \int ((\partial_{j+1}-\partial_j) f)^2d\mu,
\end{eqnarray*}
with $\alpha=2/n$, from where \eqref{eq:ec2-1} follows.
%%%%%%%%%%%%%%%%%%%%%%%%%%%%%%%%%%%%%%%%%%%%%%%%%%%%%%%%
%%%%%%%%%%%%%%%%%%%%%%%%%%%%%%%%%%%%%%%%%%%%%%%%%%%%%%%%
%%%%%%%%%%%%%%%%%%%%%%%%%%%%%%%%%%%%%%%%%%%%%%%%%%%%%%%%
%%%%%%%%%%%%%%%%%%%%%%%%%%%%%%%%%%%%%%%%%%%%%%%%%%%%%%%%
%%%%%%%%%%%%%%%%%%%%%%%%%%%%%%%%%%%%%%%%%%%%%%%%%%%%%%%%
%%%%%%%%%%%%%%%%%%%%%%%%%%%%%%%%%%%%%%%%%%%%%%%%%%%%%%%%
\appendix

\section{Construction of the dynamics}\label{sec:construction}
The system of equations \eqref{eq:main-eq} can be reformulated as
\begin{eqnarray*}
	u_j(t) = \frac12 \int^t_0 \Delta u_j(s)\, ds + \gamma \int^t_0 B_j(u(s))\, ds + \xi_j(t)-\xi_{j-1}(t).
\end{eqnarray*}
We consider the system $u^M$ on $\ze_M=\ze\slash M\ze$ evolving under its invariant distribution. 
We first check that, for all $j$ and $T>0$ 
\begin{eqnarray*}
	\esp\left[\sup_{0\leq t \leq T} |u^M_j(t)|^2\right]<\infty,
\end{eqnarray*}
so that the dynamics is well-defined.
% It is enough to control the supremum of each term on the right-hand side. 
Everything boils down to estimates of type
\begin{eqnarray*}
	\esp\left[ \sup_{0\leq t \leq T} \left| \int^t_0 u^M_j(s)\, ds \right|^2\right]
		&\leq& T \esp\left[ \sup_{0\leq t \leq T} \int^t_0|u_j^M(s)|^2ds\right]\\
		&\leq& T \esp\left[\int^T_0|u_j^M(s)|^2ds\right]\\
		&\leq& T^2,
\end{eqnarray*}
where we used invariance in the last step.

Next, we show tightness of the processes (in $M$) where we now identify $u^M$ with a periodic system on the line. This follows from Kolmogorov's criterion. It is enough to control expressions of type
\begin{eqnarray*}
	\esp\left[ \left| \int^t_s u^M_j(r)\, dr \right|^4\right]
		&\leq& |t-s|^3\esp\left[  \int^t_s \left| u^M_j(r)\right|^4\, dr \right]
		\leq C |t-s|^3.
\end{eqnarray*}
Together with a standard estimate on the increments of the Brownian motion, this yields
\begin{eqnarray*}
	\esp\left[ |u^M_j(t)-u^M_j(s)|^2\right] \leq C |t-s|^2.
\end{eqnarray*}
Hence, each coordinate is tight. By diagonalization, we can extract a subsequence of $M_k$ such that $(u_j^{M_k})$ converges in law in $C[0,T]$ for each $j$. This gives a meaning to the system \eqref{eq:main-eq}.

%%%%%%%%%%%%%%%%%%%%%%%%%%%%%%%%%%%%%%%%%%%%%%%%%%%%%%%%
%%%%%%%%%%%%%%%%%%%%%%%%%%%%%%%%%%%%%%%%%%%%%%%%%%%%%%%%
%%%%%%%%%%%%%%%%%%%%%%%%%%%%%%%%%%%%%%%%%%%%%%%%%%%%%%%%
%%%%%%%%%%%%%%%%%%%%%%%%%%%%%%%%%%%%%%%%%%%%%%%%%%%%%%%%
%%%%%%%%%%%%%%%%%%%%%%%%%%%%%%%%%%%%%%%%%%%%%%%%%%%%%%%%
%%%%%%%%%%%%%%%%%%%%%%%%%%%%%%%%%%%%%%%%%%%%%%%%%%%%%%%%

\end{document}